\documentclass[twoside,11pt]{article}

\usepackage[preprint]{jmlr2e}

\jmlrheading{1}{2023}{1-48}{4/00}{10/00}{00-000}{Attias and Kontorovich}

\ShortHeadings{Fat-Shattering Dimension
of $k$-fold Aggregations}{Attias and Kontorovich}

\usepackage{mathrsfs}
\usepackage{amsmath, amssymb}

\usepackage{mathtools}

\usepackage{xcolor}

\usepackage{stmaryrd}  %
\newcommand{\pred}[1]{\chr[#1]}
\newcommand{\nrm}[1]{\left\Vert #1 \right\Vert}

\newcommand{\E}{\mathop{\mathbb{E}}}

\newcommand{\R}{\mathbb{R}}

\newcommand{\N}{\mathbb{N}}

\newcommand{\paren}[1]{\left( #1 \right)}

\newcommand{\set}[1]{\left\{ #1 \right\}}

\newcommand{\abs}[1]{\left| #1 \right|}

\newcommand{\beq}{\begin{eqnarray*}}
\newcommand{\eeq}{\end{eqnarray*}}
\newcommand{\beqn}{\begin{eqnarray}}
\newcommand{\eeqn}{\end{eqnarray}}
\newcommand{\ben}{\begin{enumerate}}
\newcommand{\een}{\end{enumerate}}
\newcommand{\bit}{\begin{itemize}}
\newcommand{\eit}{\end{itemize}}
\providecommand{\hide}[1]{}

\newcommand{\eps}{\varepsilon}
\newcommand{\evalat}[2]{\left.#1\right|_{#2}}

\newcommand{\ceil}[1]{\ensuremath{\left\lceil#1\right\rceil}}
\newcommand{\floor}[1]{\ensuremath{\left\lfloor#1\right\rfloor}}

\newcommand{\chr}{\boldsymbol{{1}}} %

\newcommand{\X}{\Omega}

\newcommand{\fat}{\operatorname{fat}}
\newcommand{\faat}{\operatorname{\textup{f\aa t}}}

\newcommand{\med}{\operatorname{med}}
\newcommand{\mean}{\operatorname{mean}}
\newcommand{\maxmin}{\operatorname{max-min}}
\newcommand{\hinge}{\operatorname{Hinge}}

\renewcommand{\d}{\mathrm{d}}

\renewcommand{\phi}{\varphi}

\newcommand{\nrmp}[2]{\nrm{#1}_{L_{#2}(\mu)}}

\newcommand{\covn}{\mathcal{N}}

\newcommand{\nrmpk}[2]{\nrm{#1}_{L_{#2}^{(k)}(\mu)}}

\newcommand{\vc}{\operatorname{vc}}
\newcommand{\sign}{\operatorname{sign}}

\newcommand{\discr}[1]{[#1]_{\gamma}^{\star}}

\newcommand{\maavar}[2]{ 
\overset{\textup{(#1)}}{#2}
}

\newcommand{\disa}{\mathscr{D}}
\newcommand{\setzo}{\set{0,1}}
\newcommand{\setzos}{\set{0,1,\star}}
\newcommand{\st}{^\star}
\newcommand{\Log}{\operatorname{Log}}
\newcommand{\absconv}{\operatorname{absconv}}

\usepackage{cleveref}

\begin{document}

\title{%
Fat-Shattering Dimension
of $k$-fold Aggregations\thanks{A previous
version of this paper was titled ``Fat-shattering dimension of $k$-fold maxima''.}
}

 \author{\name{Idan Attias} \email{idanatti@post.bgu.ac.il}\\ 
 \name{Aryeh Kontorovich} \email{karyeh@cs.bgu.ac.il}\\   
  \addr Department of Computer Science\\
       Ben-Gurion University of the Negev, Beer Sheva, Israel
}

\editor{}

\maketitle

\begin{abstract}
We provide  estimates on the fat-shattering dimension of aggregation rules of real-valued
function classes. The latter 
consists of all ways 
of choosing
$k$ functions, one from each of the $k$
classes, and computing a pointwise function of them, such as the median, mean, and maximum.
The bound is stated in terms of the
fat-shattering dimensions of the component classes.
For linear and affine function classes,
we provide a considerably sharper upper bound
and a matching lower bound,
achieving, in particular, an optimal
dependence on $k$.
Along the way, we 
improve several known results
in addition to
pointing out
and correcting a number of erroneous
claims in the literature.
\end{abstract}


\section{Introduction}

The {\em fat-shattering dimension}, also known as
``scale-sensitive''
and
the ``parametrized variant of the $P$-dimension'',
was first proposed by \citet{183126};
its key role in learning theory
lies in characterizing the PAC learnability
of real-valued function classes
\citep{alon97scalesensitive,DBLP:journals/jcss/BartlettL98}.

In this paper, we study the 
behavior of
the fat-shattering dimension
under
various
$k$-fold aggregations.
Let $F_1,\ldots,F_k\subseteq\R^\X$ be real-valued function classes, and $G:\R^k\rightarrow \R$ be an aggregation rule. We consider the 
{\em aggregate}
function class $G(F_1,\ldots,F_k)$, which consists of all 
mappings
$x\mapsto G(f_1(x),\ldots,f_k(x))$, for any $f_1\in F_1,\ldots,f_k\in F_k$. 
Some natural aggreagation rules include
the pointwise $k$-fold maximum, median, mean, and max-min.
We seek to bound
the shattering complexity of
$G(F_1,\ldots,F_k)$
in terms of the constituent
fat-shattering dimensions of the $F_i$.
This question naturally arises in the context of ensemble methods, such as boosting and bagging, where the learner's prediction consists of an aggregation of base learners.

The analogous question regarding
aggregations of VC classes
(VC dimension being the combinatorial
complexity controlling
the learnability of Boolean function classes) have been studied
in detail and largely resolved
\citep{78414,
MR1072253,
DBLP:journals/ipl/EisenstatA07,
DBLP:journals/ipl/Eisenstat09,
DBLP:journals/jmlr/CsikosMK19}.
Furthermore, 
closure properties were also studied in the context of online classification and private PAC learning \citep{alon2020closure,ghazi2021near} for the Littlestone and Threshold dimensions. However, for real-valued functions, this question remained largely uninvestigated.

Our contributions are as follows:
\begin{itemize}
    \item For a natural class of aggregation rules that commute with shifts (see Definition (\ref{def:commutes-shift})),  assuming 
    $\fat_\gamma(F_i)\leq d$, for $1\leq i \leq k$, we show that
    \begin{align*}
    \fat_\gamma(G(F_1,\ldots,F_k))
    \leq O\left(dk\log^2\left(dk\right)\right),
    \qquad
    \gamma>0.
    \end{align*}
The formal statement is given in
\Cref{thm:max-fat}. 
    By using an entirely different approach, for aggregations that are $L$-Lipschitz ($L\geq 1$) in supremum norm (see Definition (\ref{def:lip}))
    and for bounded function classes $F_1,\ldots,F_k \subset [-R,R]^\Omega$ with 
    $\fat_{\eps\gamma}(F_i)\leq d$, for $1\leq i \leq k$, we show that
    \begin{align*}
    \fat_\gamma(G(F_1,\ldots,F_k))
    \leq O\left(dk\Log^{1+\epsilon}
    \frac{LRk}{\gamma}
    \right),
    \qquad
    0<\gamma/L<R \text{ and }0<\eps<\log 2,
    \end{align*}
    where $\Log(x):=\log(e\vee x)$.
The formal statement is given in
\Cref{thm:max-fat-RV}. 
    
    In particular, our proofs hold for the maximum, minimum, median, mean, and max-min aggregations.
    \item For $R$-bounded affine functions and for aggregations that are $L$-Lipschitz in supremum norm, we show the following dimension-free bound,
    \beq
\fat_{\gamma}(G(F_1,\ldots,F_k))
&\le&
O\left(
\frac{
L^2R^2 k\log(k)
}{
\gamma^2
}
\right)
,
\qquad
0<\gamma/L<R
.
\eeq
This result also extends to the hinge-loss class of affine functions. The formal statement is given in
\Cref{thm:max-fat-lin-R}.
In particular, we improve by a log factor 
the estimate of
\citet[Lemma 6]{fefferman2016testing} on the fat-shattering dimension of max-min aggregation of
linear functions.

Furthermore, in Corollary \ref{thm:rademacher} we show an upper bound on the Rademacher complexity of the $k$-fold maximum aggregation of affine functions and hinge-loss affine functions. Our bound scales with $\sqrt{k}$, improving upon \citet{DBLP:journals/jmlr/RavivHO18} where the dependence on $k$ is linear.
\item For affine functions and the 
$k$-fold max
aggregation, we show tight dimension-dependent bounds (up to constants),
\beq
\fat_\gamma(
G_{\max}(F_1,\ldots,F_k))
&=&
\Theta\left(dk\log k\right)
,
\qquad
\gamma>0,
\eeq
where $d$ is the Euclidean dimension. For the formal statements, see \Cref{thm:max-fat-lin-d,thm:logk-lb}.
\end{itemize}

\paragraph{Applications.}
The need to analyze the combinatorial complexity of a $k$-fold maximum of function classes (see (\ref{eq:Fmax}) for the formal definition) arises in a number of diverse settings. 
One natural example is adversarially robust PAC learning to test time attacks for real-valued functions \citep{attias2022improved, attias2023adversarially}. In this setting, the learner observes an i.i.d. labeled sample from an unknown distribution, and the goal is to output a hypothesis with a small error on unseen examples from the same distribution, with high probability. The difference from the standard PAC learning model is that at test time, the learner only observes a corrupted example, while the prediction is tested on the original label. 
Formally, $(x,y)$ is drawn from the unknown distribution, there is an adversary that can map $x$ to $k$ possible corruptions $z$ that are known to the learner. The learner observes only $z$ while its loss is with respect to the original label $y$.
This scenario is naturally captured by the $k$-fold max:
the learner aims to learn the maximum aggregation of the loss classes. \citet{attias2022improved} showed that uniform convergence holds in this case, and so the sample complexity of an empirical risk minimization algorithm is determined by the complexity measure of the $k$-fold maximum aggregation.

Analyzing the $k$-fold maximum arises also in a setting of learning polyhedra with a margin.
\citet{GKKN-nips18} provided a learning algorithm that
represents polyhedra as intersections of bounded affine functions. The sample complexity of the algorithm is determined by the complexity measure of the maximum aggregation of affine function classes.

Another natural example of where the $k$-fold maximum and $k$-fold max-min play a role is in analyzing the convergence of $k$-means clustering. \citet{fefferman2016testing} bounded the max-min aggregation and \citet{klochkov2021robust,biau2008performance,appert2021new,zhivotovskiy2022sharp} bounded the max aggregation.
The main challenge in this setting is bounding the covering numbers of the aggregation over $k$ function classes which can be obtained by bounding the Rademacher complexity or the fat-shattering dimension.

Finally, there are numerous ensemble methods for regression that output some aggregation of base learners, such as the median or mean. 
Examples of these methods include boosting (e.g., \citet{freund1997decision,kegl2003robust}), bagging (bootstrap aggregation) by \citet{breiman1996bagging}, and its extension to the random forest algorithm \citep{breiman2001random}.

\paragraph{Related work.}\label{sec:existing-res}
It was claimed in
\citet[Theorem 12]{DBLP:conf/alt/AttiasKM19}
that
$
\fat_\gamma(
G_{\max}
)
\le
2\log(3k)\sum_{j=1}^k 
\fat_{\gamma}(F_i),
$
but the proof had a mistake (see Section~\ref{sec:discussion});
our 
Open Problem
asks if the general
form of the bound
does holds
(we believe it does at least for the max aggregation).
Using the recent
disambiguation result
of
\citet{DBLP:journals/corr/abs-2107-08444}
presented in
Lemma~\ref{lem:disambig} here,
\citet[Lemma 15]{attias2022improved}
obtained the
bound 
$
\fat_\gamma(
G_{\max})
\le
O\paren{\Log(k)\Log^2(|\X|)\sum_{j=1}^k\fat_\gamma(F_i)}$, where $\Omega$ is the domain of the function classes $F_1,\ldots,F_k$.
The latter is, in general, incomparable
to Theorem~\ref{thm:max-fat}
--- but is clearly
a considerable improvement for large
or infinite $\X$.

Using 
the covering number
results of 
\citet{MR1965359,talagrand2003vapnik}
(see Section~\ref{sec:cov-num}),
\citet[Theorem 6.2]{duan11} obtained a general result, which, when specialized to
$k$-fold maxima, yields
\beqn
\label{eq:duan}
\fat_{\gamma}(G_{\max})
&\le&
O
\left(
{\Log\frac{
k}{\gamma}}
\cdot
\sum_{i=1}^k
\fat_{c\gamma/\sqrt k}(F_i)
\right)
\eeqn
for a universal constant
$c>0$;
(\ref{eq:duan})
is an immediate consequence of
Theorem~\ref{thm:max-cov} (with $p=2$),
Lemma~\ref{lem:talag},
and
Lemma~\ref{lem:MV} in this paper.
Our results improve
over (\ref{eq:duan}) by
removing the dependence on $k$
in the scale of the fat-shattering
dimensions; however, \citeauthor{duan11}'s
general method is applicable to
a wider class of uniformly continuous $k$-fold aggregations.

\citet[Lemma A.2]{srebro2010smoothness}
bounded the fat-shattering dimension
in terms of the Rademacher complexity.
\citet{DBLP:journals/corr/abs-1911-06468}
bounded the Rademacher complexity
of a smooth $k$-fold aggregate, see
also references therein. 
Inspired by
\citet{appert2021new},
\citet{zhivotovskiy2022sharp} has obtained the best known upper bound on the Rademacher complexity of $k$-fold maximum over linear function classes.
\citet{DBLP:journals/jmlr/RavivHO18} upper bounded the Rademacher complexity
of the $k$-fold maximum aggregation of affine functions and hinge-loss affine functions.

\section{Preliminaries}\label{sec:prelim}

\paragraph{Aggregation rules.} 
A $k$-fold {\em aggregation} rule is any 
mapping
$G:\R^k\rightarrow \R$.
Just as $G$ maps $k$-tuples of reals into reals,
it naturally aggregates $k$-tuples of functions
into a single one:
for
$f_1,\ldots,f_k:\Omega\rightarrow \R$,
we define
$G(f_1,\ldots,f_k):\Omega\rightarrow \R$ as the mapping $x \mapsto G(f_1(x),\ldots,f_k(x))$.
Finally, the aggregation extends to $k$-tuples of function classes:
for
$F_1,\ldots,F_k\subseteq\R^\Omega$,
we define
\begin{equation}\label{def:aggregation}
 G(F_1,\ldots,F_k):=
 \set{x\mapsto G(f_1(x),\ldots,f_k(x)):f_i \in F_i
,
i\in[k]
}.
\end{equation}
A canonical example of an aggregation rule is the $k$-fold max,
induced by the mapping 
\begin{align}\label{eq:Fmax}
G_{\max}(x_1,\ldots,x_k)
:=
\max_{i\in[k]}x_i
.  
\end{align}

Next, we consider some properties that an aggregation
rule might possess.

\noindent\textit{Commuting with shifts.} 
We say that an aggregation rule $G$ commutes with shifts if 
\begin{equation}\label{def:commutes-shift}
G(x)-r
=
G(x-r),
\qquad x\in\R^k, r\in\R
.
\end{equation}

\noindent\textit{Lipschitz 
continuity.}
The mapping $G:\R^k\to\R$
is $L$-Lipschitz
with respect to
$\nrm{\cdot}_\infty$ if
\beqn
\label{def:lip}
\abs{G(x)-G(x')}
\le
L\nrm{x-x'}_\infty
=
L\max_{i\in[k]}|x_i-x_i'|
,
\qquad
x,x'\in\R^k.
\eeqn

Many natural aggregation rules possess these properties, such as maximum, minimum, median, mean, and max-min. 
The maximum is defined in (\ref{eq:Fmax}), the minimum is defined  analogously as
\begin{align*}
   G_{\min}(x_1,\ldots,x_k)
   :=
\min_{i\in[k]}x_i,
\end{align*}
the median is defined as 
\begin{align*}
  G_{\med}(x_1,\ldots,x_k)
   :=
\begin{cases}
x_{\frac{k+1}{2}}, & k \text{ is odd},\\
\frac{1}{2}(x_{\frac{k}{2}}+x_{\frac{k+1}{2}}), & k \text{ is even},
\end{cases}  
\end{align*}
and
the mean is defined as 
\begin{align*}
G_{\mean}
(x_1,\ldots,x_k)
   :=
\frac{1}{k}\sum^k_{i=0} x_i.
\end{align*}
We also define
$
G_{\maxmin}
:\R^{\ell\times k}
\to\R
$
as 
\begin{align}
\label{eq:maxmindef}
G_{\maxmin}
(x_{11},\ldots,x_{k\ell})
   :=
   \max_{j\in [\ell]}\min_{i\in[k]}x_{ij};
\end{align}
it is readily verified
to commute with shifts
and Lemma~\ref{lem:max-min}
shows that it is $1$-Lipschitz
with respect to $\nrm{\cdot}_{\infty}$.

\paragraph{Fat-shattering dimension.}
Let $\X$ be a set and $F\subset\R^\X$. For $\gamma>0$,
a set $S=\set{x_1,\ldots,x_m}\subset\X$
is said to be $\gamma$-shattered by
$F$
if
\beqn
\label{eq:gamma-shatter}
\sup_{r\in\R^m}
\;
\min_{y\in\set{-1,1}^m}
\;
\sup_{f\in F}
\;
\min_{i\in[m]}
\;
y_i(f(x_i)-r_i)\ge \gamma.
\eeqn
The $\gamma$-fat-shattering dimension, denoted by $\fat_\gamma(F)$,
is the size of the largest $\gamma$-shattered set (possibly $\infty$).

\paragraph{Fat-shattering dimension at zero.}
As in
\citet{DBLP:journals/tit/GottliebKK14+colt},
we also define the notion
of $\gamma$-shattering at $0$,
where the ``shift'' $r$ 
in (\ref{eq:gamma-shatter})
is constrained to be $0$. Formally, the shattering condition
is
$
\min_{y\in\set{-1,1}^m}
\sup_{f\in F}
\min_{i\in[m]}
y_if(x_i)\ge \gamma
,
$
and the 
we denote
corresponding 
dimension 
by
$\faat_\gamma(F)$.

\citet[Lemma 13]{DBLP:conf/alt/AttiasKM19}
showed that for all $F\subset\R^\X$,
\beqn
\label{eq:fatfat0}
\fat_\gamma(F)
=\max_{r\in\R^\X}\faat_\gamma(F-r),
\qquad
\gamma>0,
\eeqn
where $F-r=\set{f-r; f\in F}$
is the $r$-shifted class (the maximum is 
always achieved).
Lemma~\ref{lem:fat-faat}
presents another, apparently novel, connection between $\fat$
and $\faat$.

\paragraph{Rademacher complexity.}
Let $\mathcal{F}$ be of real-valued function class on the domain space $\mathcal{W}$. Define the empirical Rademacher complexity of $\mathcal{F}$ on a given sequence $(w_1,\ldots,w_n)\in \mathcal{W}^n$ is defined as
\begin{align*}
    \mathcal{R}_n(\mathcal{F} | w_1,\ldots,w_n)
=
\mathbb{E}_{\mathbf{\sigma}} \sup_{f\in \mathcal{F}}\frac1n\sum_{i=1}^n\sigma_i f(w_i),
\end{align*}
where $\sigma=(\sigma_1,\ldots\sigma_n)$ are independent random variables uniformly chosen from $\set{-1,1}$.
The Rademacher complexity of $\mathcal{F}$ with respect to a distribution $\mathcal{D}$ is defined as 
\begin{align*}
\mathcal{R}_n(\mathcal{F})
=
\mathbb{E}_{w_1,\ldots,w_n\sim \mathcal{D}}R_n(\mathcal{F} |w_1,\ldots,w_n).
\end{align*}
It is a classic fact
\citep[Theorem 3.1]{mohri-book2012}
that the Rademacher complexity controls generalization
bounds in a wide range of supervised learning settings.

\paragraph{Covering numbers.}
We start with some background on covering numbers. 
Whenever $\X$ is endowed with a probability measure $\mu$, this induces, 
for $p\in[1,\infty)$ and $f:\Omega\rightarrow \R^k$,
the 
norm 
\beq
\nrmpk{f}{p}^p
=
\E_{X\sim \mu}\nrm{f(X)}_p^p
=
\int_\X \nrm{f(x)}_p^p\d\mu(x)
\eeq
on
$L^{(k)}_p(
\mu):=\set{
f\in(\R^k)^\X:
\nrmpk{f}{p}<\infty
}
$.
When $k=1$, we 
write
$
L_p(
\mu):=
L^{(1)}_p(
\mu)
$.
For $p=\infty$,
$\nrmpk{f}{\infty}$
is the essential supremum of $f$
with respect to $\mu$.
For 
$t>0$
and
$H\subset F\subset L_p(
\mu)$,
we say that $H$ is a $t$-cover
of $F$ under $\nrmp{\cdot}{p}$
if
$
\sup_{f\in F}\inf_{h\in H}\nrmp{f-h}{p}
\le t.
$
The $t$-covering number of $F$,
denoted by $\covn(F,
L_p(\mu),
t)$,
is the cardinality of the smallest $t$-cover
of $F$ (possibly, $\infty$).
We note the obvious relation
\beqn
\label{eq:covn-pq}
p>q
&\implies&
\covn(F,L_p(\mu),t)
\ge
\covn(F,L_q(\mu),t),
\eeqn
which holds for all probability measures $\mu$
and all $t>0$.

We sometimes overload the notation about aggregations by defining $G$ on $k$-tuples of functions (instead on  $k$-tuples of reals), $G: (\R^{\Omega})^k\rightarrow \R^{\Omega}$. 
We say that $G$ is $L$-Lipschitz with respect to $\nrmpk{\cdot}{p}$, if 
\beq
\nrmp{G(f_{1:k})-G(f'_{1:k})}{p}
\le
L\nrmpk{(f_{1:k})-(f'_{1:k})}{p}
,
\qquad
f_{1:k},f'_{1:k}\in(\R^k)^\X.
\eeq

\paragraph{Notation.}
We write $\N=\set{0,1,\ldots}$ to denote the natural numbers.
For $n\in\N$, we write
$[n]:=\set{1,2,\ldots,n}$.
All of our logarithms are base $e$, unless explicitly denoted otherwise.
We use $\max\set{u,v}$
and $u\vee v$ interchangeably,
and write
$\Log(x):=\log(e\vee x)$.
For any function class $F$
over a set $\X$ and $E\subset\X$,
$
F(E)=\evalat{F}{E}
$
denotes the projection 
on (restriction to) $E$.
In line with the common convention
in functional analysis,
absolute numerical constants
will be denoted by letters such as
$C,c$,
whose value
may change from line to line.
Any transformation $\phi:\R\to\R$
may be applied to a function $f\in\R^\X$
via $\phi(f):=\phi\circ f$,
as well as to $F\subset\R^\X$
via $\phi(F):=\set{\phi(f);f\in F}$.
The sign function thresholds at $0$:
$\sign(t)=\pred{t\ge0}$.

\section{Main Results}
Our main results involve upper-bounding
the fat-shattering dimension
of aggregation rules in terms
of the dimensions of the component classes.
We begin with the simplest (to present):
\begin{theorem}[General function classes and aggregations that commute with shifts]
\label{thm:max-fat}
For 
$F_1,\ldots,F_k\subseteq\R^\X$,
and aggregation rule
$G$ that commutes with shifts, (see definition (\ref{def:commutes-shift})),
we have
\beq
\label{eq:max-fat}
\fat_\gamma(G(F_1,\ldots,F_k))
&\le&
25
D_\gamma \log^2(
90
D_\gamma),
\qquad
\gamma>0,
\eeq
where
$D_\gamma:=\sum_{i=1}^k\fat_\gamma(F_i)>0$.
In the degenerate case
where $D_\gamma=0$,
$\fat_\gamma(G)=0$. 

In particular, this result holds for natural aggregation rules, such as maximum, minimum, median, and mean.
\end{theorem}
\paragraph{Remark.}
We made no attempt to optimize the constants;
these are only provided to give a rough
order-of-magnitude sense. In the sequel,
we forgo numerical estimates 
and state the results in terms of
unspecified universal constants.

The next result 
provides an alternative bound
based on an entirely
different technique:

\begin{theorem}[Bounded function classes and Lipschitz aggregations]
\label{thm:max-fat-RV}
For 
$0<\eps<\log 2$,
$F_1,\ldots,F_k\subseteq
[-R,R]
^\X$,
and an aggregation rule
$G$ that is $L$-Lipschitz ($L\geq 1$) in supremum norm (see definition (\ref{def:lip})),
we have
\beq
\label{eq:max-fat-RV}
\fat_\gamma(G(F_1,\ldots,F_k))
&\le&
CD\Log^{1+\eps}\frac{LRk}{\gamma}
,
\qquad
0<\gamma/L<R,
\eeq
where
\beq
D &=& 
\sum_{i=1}^k\fat_{c\eps \gamma}(F_i)
\eeq
and $C,c>0$ are universal constants.
\end{theorem}
In \Cref{sec:cov-num}, we show that maximum, median, mean, and max-min aggregations are $1$-Lipschitz.
\paragraph{Remark.}
The bounds in Theorems~\ref{thm:max-fat}
and \ref{thm:max-fat-RV} are, in general, incomparable---and not just because of the unspecified constants in the latter. One notable difference is that
Theorem~\ref{thm:max-fat} only depends on the
shattering scale $\gamma$,
while 
Theorem~\ref{thm:max-fat-RV} additionally
features a (weak) explicit dependence on the aspect
ratio $R/\gamma$. 
In particular,
Theorem~\ref{thm:max-fat} is applicable to semi-bounded affine classes
(see Section~\ref{sec:affine}),
while Theorem~\ref{thm:max-fat-RV} is not.
Still, for fixed $R,\gamma$
and large $k$, 
the latter
presents a significant asymptotic
improvement over the former.

For the special case of affine functions and hinge-loss affine functions, 
the technique of Theorem~\ref{thm:max-fat-RV}
yields a
considerably sharper estimate:
\begin{theorem}[Dimension-free bound for Lipschitz aggregations of affine functions]
\label{thm:max-fat-lin-R}
Let $B\subset\R^d$ be the $d$-dimensional
Euclidean unit ball 
and
\beqn
\label{eq:Fi-max-fat-lin-R}
F_i=\set{x\mapsto w\cdot x+b;
\nrm{w}\vee |b|\le R_i, w\in \R^d, b,R_i\in \R},
\qquad
i\in[k],
\eeqn
be $k$ collections of 
$R_i$-bounded affine functions
on $\X=B$
and
$G$ be 
an
aggregation rule that is $L$-Lipschitz in supremum norm (see definition (\ref{def:lip})).
Then
\beqn
\label{eq:Fi-max-fat-lin-R-bd}
\fat_{\gamma}(G(F_1,\ldots,F_k))
&\le&
\frac{
CL^2\Log(k)
}{
\gamma^2
}
\sum_{i=1}^k 
{R_i^2}
,
\qquad
0<\gamma/L<\min_{i\in[k]}
R_i
,
\eeqn
where $C>0$
is a universal constant.
Further, if
\begin{align}\label{hinge}
    F_i^{\hinge}=\set{(x,y)\mapsto 
\max\set{0,1-yf(x,y)}
; f\in F_i}
\end{align}
is a family
of
$R_i$-bounded hinge-loss affine functions
for $i\in[k]$
and
$G_{\hinge}\equiv G(F_1^{\hinge},\ldots,F_k^{\hinge})$ 
is 
an
aggregation rule that is $L$-Lipschitz in supremum norm, then the same bound as in (\ref{eq:Fi-max-fat-lin-R-bd}) hold for $\fat_{\gamma}(G_{\hinge})$.
\end{theorem}

In particular, Theorem \ref{thm:max-fat-lin-R} improves by a log factor 
the estimate of
\citet{fefferman2016testing},
on
the fat-shattering dimension of max-min aggregation (defined in Section \ref{sec:prelim}) of linear functions:\footnote{
The max-min aggregation is 
shown to be
$1$-Lipschitz in supremum norm 
in
Lemma \ref{lem:max-min} of Section \ref{sec:cov-num}.}
\begin{lemma}[\citet{fefferman2016testing}, Lemma 6]
Let $B\subset\R^d$ be the $d$-dimensional
Euclidean unit ball 
and
\beq
F_{ij}=\set{x\mapsto w\cdot x;
\nrm{w}\leq \nrm{1}, w\in \R^d},
\qquad
i\in[k],j\in[\ell],
\eeq
be $k\ell$ (identical)
linear function classes defined on $\X=B$.
If
$G_{\maxmin}$ is the
max-min aggregation rule \eqref{eq:maxmindef}, then
\beq
\fat_\gamma(
G_{\maxmin}(F_{11},\ldots,F_{k\ell})
)
&\le&
C\frac{k\ell}{\gamma^2}\log^2\left(C\frac{k\ell}{\gamma^2}\right),
\eeq
where $C>0$ is a universal constant.
\end{lemma}
Our
Theorem \ref{thm:max-fat-lin-R} 
improves the latter by a log factor:
\beq
\fat_\gamma(
G_{\maxmin}(F_{11},\ldots,F_{k\ell})
)
&\le&
C\frac{k\ell\log\left(k\ell\right)}{\gamma^2}.
\eeq
\begin{corollary}[Rademacher complexity for $k$-Fold Maximum of Affine Functions]\label{thm:rademacher}
Let $F_i$ be an $R_i$-bounded affine function 
class
as in (\ref{eq:Fi-max-fat-lin-R}) or a hinge loss affine function as in (\ref{hinge}), let
$G_{\max}$ be the maximum aggregation rule, and let $\tilde{R}=\max_i R_i$, then
\begin{align*}
    \mathcal{R}_n(G_{\max}(F_1,\ldots,F_k))
    \leq
C\sqrt{\frac{\Log(k)\log^3(\tilde{R}n)\tilde{R}^2\sum_{i=1}^k 
{R_i^2}}{n}}.
\end{align*}
where $\mathcal{R}_n$ is the Rademacher complexity and $C>0$ is a universal constant.
\end{corollary}
Corollary \ref{thm:rademacher} improves upon \citet[Theorem 7]{DBLP:journals/jmlr/RavivHO18}. Their upper bound scales linearly with $k$, whereas ours 
as $\sqrt{k\log k}$.

Note, however, that
for linear classes
a better
bound 
is known:
\begin{theorem}[\citet{zhivotovskiy2022sharp}]
Let $B\subset\R^d$ be the $d$-dimensional
Euclidean unit ball 
and
\beq
F_i=\set{x\mapsto w\cdot x;
\nrm{w}\leq {1}, w\in \R^d},
\qquad
i\in[k]
\eeq
be $k$ (identical)
linear function classes defined on $\X=B$.
If
$G_{\max}$ is the
maximum aggregation rule, then
\begin{align*}
    \mathcal{R}_n(G_{\max}(F_1,\ldots,F_k))
    \leq
    C \log\left(\frac{n}{k}\right)\sqrt{\frac{k\log k}{n}},
\end{align*}
where $\mathcal{R}_n$ is the Rademacher complexity and $C>0$ is a universal constant.
\end{theorem}

The estimate in \Cref{thm:max-fat-lin-R} is {\em dimension-free} in the sense of being independent of $d$. In applications where
a dependence on $d$ is admissible, 
an
optimal bound can be obtained:
\begin{theorem}[Dimension-dependent bound for $k$-fold maximum of affine functions]
\label{thm:max-fat-lin-d}
Let
$\X=\R^d$
and
$F_i\subset\R^\X$ be $k$ (identical)
function classes consisting of all
real-valued affine functions:
\beq
F_i=\set{x\mapsto w\cdot x+b; w\in \R^d, b\in \R
},
\qquad
i\in[k]
\eeq
and let $G_{\max}$ be their $k$-fold maximum (see definition (\ref{eq:Fmax})).
Then
\beq
\fat_\gamma(
G_{\max}(F_1,\ldots,F_k)
)
&\le&
Cdk\Log k
,
\qquad
\gamma>0,
\eeq
where $C>0$ is a universal constant.
\end{theorem}

The 
optimality of the
upper bound in
Theorem~\ref{thm:max-fat-lin-d}
is witnessed by the
matching lower bound:
\begin{theorem}[Dimension-dependent lower bound for $k$-fold maximum of affine functions]
\label{thm:logk-lb}
For $k\ge1$ and $d\ge4$,
let
$F_1=F_2=\ldots=F_k$ be the collection of all
affine functions over $
\X=
\R^d$
and let $G_{\max}$ 
be their $k$-fold maximum (see definition (\ref{eq:Fmax})).
Then
\beq
\fat_\gamma(
G_{\max}(F_1,\ldots,F_k)
)
&\ge&
C\log(k)
\sum_{i=1}^k
\fat_\gamma(F_i)
=
Cdk\log k
,
\qquad
\gamma>0,
\eeq
where $C>0$ is a universal constant.
\end{theorem}
The scaling argument employed in the
proof of Theorem~\ref{thm:logk-lb}
can be invoked to show that
the claim continues to hold
for $\X=B$.

Together, Theorems~\ref{thm:max-fat-lin-d}
and \ref{thm:logk-lb} show that
the logarithmic
dependence on $k$
is optimal.

\section{Proofs}

We start with upper-bounding the fat-shattering dimension of aggregation rule that commute with shifts, in terms of the dimensions of the component classes.
\subsection{Proof of Theorem~\ref{thm:max-fat}}
\label{sec:max-fat}

\paragraph{Partial concept classes and disambiguation.}

We say that $F\st\subseteq\setzos^\X$
is a 
{\em partial}
concept class over
$\Omega$;
this usage is consistent with
\citet{DBLP:journals/corr/abs-2107-08444},
while \citet{DBLP:conf/alt/AttiasKM19,attias2022improved}
used the descriptor 
{\em ambiguous}.
For $f\st\in
F\st
$,
define its {\em disambiguation set}
$\disa(f\st)\subseteq\setzo^\X$
as
\beq
\disa(f\st)=
\set{
g\in\setzo^\X:
\forall x\in\X,~f\st(x)\neq\star\implies f\st(x)=g(x)
};
\eeq
in words, $\disa(f\st)$ consists of the
{\em total} concepts $g:\X\to\setzo$
that agree pointwise with $f\st$, whenever the latter takes a value in $\setzo$.
We say that $\bar F\subseteq\setzo^\X$
disambiguates $F\st$
if for all $
f\st\in
F\st$,
we have $\bar F\cap\disa(f\st)\neq\emptyset$;
in words, every 
$f\st\in
F\st$
must have a disambiguated representative in
$\bar F$.\footnote{\citet{attias2022improved}
additionally required that $\bar F
\subseteq
\bigcup_{f\st\in F\st}\disa(f\st)$,
but this is an unnecessary restriction,
and does not affect any of the results.}

As in \citet{DBLP:journals/corr/abs-2107-08444,attias2022improved}, we say\footnote{
\citet{DBLP:conf/alt/AttiasKM19}
had 
incorrectly
given
$F\st(S)=\setzo^S$
as the shattering condition.
} 
that
$S\subset\X$ is VC-shattered by $F\st$
if $F\st(S)\supseteq\setzo^S$. We write $\vc(F\st)$ to denote
the size of the largest VC-shattered set
(possibly, $\infty$).
The obvious relation
$
\vc(F\st)
\le
\vc(\bar F)
$
always holds
between a partial concept class
and any of its disambiguations.
\citet[Theorem 13]{DBLP:journals/corr/abs-2107-08444}
proved the following variant of the
Sauer-Shelah-Perles Lemma for partial
concept classes:
\begin{lemma}[\citet{DBLP:journals/corr/abs-2107-08444}]
\label{lem:disambig}
For every $F\st\subseteq\setzos^\X$
with $d=\vc(F\st)<\infty$ and $|\Omega|< \infty$,
there is an $\bar F$
disambiguating $F\st$
such that
\beqn
\label{eq:disamb}
|\bar F(\X)|
&\le&
(|\X|+1)^{(d+1)\log_2|\X|+2}.
\eeqn
For $d>0$ and $|\X|>1$, this 
implies the somewhat more wieldy
estimate\footnote{
The estimate (\ref{eq:disamb-c})
does not appear in 
\citet{DBLP:journals/corr/abs-2107-08444},
but is an elementary consequence of
(\ref{eq:disamb}).} to
\beqn
\label{eq:disamb-c}
|\bar F(\X)|
&\le&
|\X|^{5d\log_2|\X|}.
\eeqn
\end{lemma}

We will make use of the 
elementary fact
\beq
\label{eq:log2x}
x\le A\log_2 x
&\implies&
x\le 3 A\log(3 A)
,
\qquad x,A\ge 1
\eeq
and its corollary
\beqn
\label{eq:log2y}
y\le A(\log_2y)^2
&\implies&
y\le 5 A\log^2(18 A)
,
\qquad y,A\ge 1.
\eeqn

\begin{proof}[of Theorem~\ref{thm:max-fat}]
We follow the basic techniques of discretization and $r$-shifting, employed in
\citet{DBLP:conf/alt/AttiasKM19,attias2022improved}. Fix $\gamma>0$
and define the operator
$\discr{\cdot}:\R\to\set{0,1,\star}$
as
\beq
\discr{t}
&=&
\begin{cases}
0, & t\le -\gamma \\
1, & t\ge\gamma \\
\star, & \text{else}.
\end{cases}
\eeq
Observe that for all
$F\subseteq\R^\X$
and
$\discr{F}:=\set{\discr{f};f\in F}$,
we have
$
\faat_\gamma(F)=\vc(\discr{F}).
$
 Let $G:\R^k\rightarrow \R$ be 
 a $k$-fold
aggregation rule
and 
$F_1,\ldots,F_k\subseteq\R^\Omega$ be real-valued function classes.
Suppose that some
$S=\set{x_1,\ldots,x_\ell}\subset\X$
is $\gamma$-shattered by $G\equiv G(F_1,\ldots,F_k)$.
Proving the claim amounts to 
upper-bounding $\ell$ appropriately.
By (\ref{eq:fatfat0}),
there is an $r\in\R^\X$
such that $\fat_\gamma(G)
=\faat_\gamma(G-r)
=\vc(\discr{G-r})
$.
Put $F'_i:=F_i-r$
and since $G$ commutes with $r$-shift, as defined in (\ref{def:commutes-shift}), we have
\beqn
\label{eq:Fmax'}
G'
:=
G(F'_1,\ldots,F'_k)
=
G(F_1-r\ldots,F_k-r)
=
G(F_1\ldots,F_k)-r
.
\eeqn
Hence, $S$ is VC-shattered
by $\discr{G'}$
and 
\beqn
\label{eq:di}
v_i:=
\vc(\discr{F_i'})
=
\faat_\gamma(F_i')
\le
\fat_\gamma(F_i')
=
\fat_\gamma(F_i)
,
\qquad
i\in[k].
\eeqn

Let us assume for now that
each $v_i>0$;
in this case,
there is no loss of
generality in assuming $\ell>1$.
Let
$\bar F_i$
be a ``good'' disambiguation
of $\discr{F_i'}$
on $S$, as furnished by
Lemma~\ref{lem:disambig}:
\beq
|\bar F_i(S)|
&\le&
\ell^{5
v_i
\log_2\ell}.
\eeq

Observe that $
{\bar G}
:=G(\bar F_1,\ldots,\bar F_k)$ is a valid disambiguation of $\discr{G'}$.
It follows that
\beqn
\label{eq:2^ell}
2^\ell 
\;=\;
|{\bar G}
(S)| 
\;\le\;
\prod_{i=1}^k |\bar F_i(S)| 
\;\le\;
\ell^{
5\log_2\ell
\sum_{i=1}^k
v_i
}
.
\eeqn
Thus,
(\ref{eq:log2y}) 
implies
that 
$
\ell\le 
25
(\sum v_i)\log^2(
90
\sum v_i)
$,
and the latter is an
upper bound on
$\vc(
{\bar G}
)
$
---
and hence, also on
$\vc(
\discr{G'}
)
=\fat_\gamma(G)
$.
The claim now follows from~(\ref{eq:di}).

If any one given $v_i=0$, we claim that (\ref{eq:2^ell}) is unaffected.
This is because
any $C\st\subset\setzos^\X$
with $\vc(C\st)=0$ has a 
singleton
disambiguation
$\bar C=\set{c}$.
Indeed, any 
given
$x\in\X$ can 
receive at most one of
$\setzo$ 
as a label
from the members of
$C$ (otherwise, it would be shattered, forcing 
$\vc(C\st)\ge1$).
If {\em any} $c\st\in C\st$ labels $x$
with $0$, then {\em all} members of $C\st$
are disambiguated to label $x$ with $0$
(and, {\em mutatis mutandis}, $1$).
Any $x$ labeled with $\star$ by {\em every} $c\st\in C\st_i$ can be disambiguated arbitrarily (say, to $0$).
Disambiguating the degenerate
$\discr{F_i'}$
to the singleton
$
\bar F_i(S)
$
has no effect on the product in
(\ref{eq:2^ell}).

The foregoing argument continues to hold
if more than one $v_i=0$. In particular,
in the degenerate case where
$\fat_\gamma(F_1)=\fat_\gamma(F_2)=\ldots=\fat_\gamma(F_k)=0$,
we have $\prod|\bar F_i(S)|=1$,
which forces $\ell=0$.
\end{proof}

\subsection{Proof of Theorem~\ref{thm:max-fat-RV}}

First, we upper bound the covering numbers of Lipschitz aggregations as a function of the covering of the component classes.

\begin{theorem}[Covering number of $L$-Lipschitz aggregations]
\label{thm:max-cov}
Let $t>0$,
$p\in[1,\infty]$,
and $F_1,\ldots,F_k\subset L_p(
\mu)$.
Let $G$ be an aggregation rule that is $L$-Lipschitz.
Then, for
all probability measures $\mu$
on $\X$,
\beq
\covn(G(F_1,\ldots,F_k),
L_p(\mu),
t)
&\le&
\begin{cases}
\prod_{i=1}^k
\covn(F_i,t/Lk^{1/p}),
&p<\infty\\
\prod_{i=1}^k
\covn(F_i,t/L),
&p=\infty.
\end{cases}
\eeq
\end{theorem}
We proceed to the main proof.
\begin{proof} [of Theorem~\ref{thm:max-fat-RV}].
 Let $G:\R^k\rightarrow \R$ be an aggregation rule that is $L$-Lipschitz ($L\geq 1$) in supremum norm, as defined in (\ref{def:lip}), and let $F_1,\ldots,F_k\subseteq[-R,R]^\Omega$ be real-valued function classes.
Suppose that some
$\X_\ell=\set{x_1,\ldots,x_\ell}\subset\X=B$
is $\gamma$-shattered by $G$,
and
let $F_i(\X_\ell)=\evalat{F_i}{\X_\ell}$.
We upper bound the covering number with the fat-shattering dimension as in Lemma~\ref{lem:RV} (see Section \ref{sec:cov-num-fat}), with $n=\ell$ and $p=
\infty$,
\beq
\log\covn(F_i(\X_\ell),L_\infty(\mu_\ell),\gamma)
&\le&
Cv_i\log(R\ell/v_i\gamma)\log^\eps(\ell/v_i),
\qquad
0<\gamma<R,
\eeq
where
$v_i=\fat_{c\eps \gamma}(F_i)$.
Then Theorem~\ref{thm:max-cov} implies that
\beq
\log\covn(G(\X_\ell),L_\infty(\mu_\ell),
\gamma/2
)
&\le&
\sum_{i=1}^k
\log\covn(F_i(\X_\ell),L_\infty(\mu_\ell),
\gamma/2L)
\\
&\le&
C\sum_{i=1}^k
v_i\log(LR\ell/v_i\gamma)\log^\eps(\ell/v_i)
\\
&\maavar{a}{\le}&
C\sum_{i=1}^k
v_i\log^{1+\eps}(LR\ell/v_i\gamma)
\\
&\maavar{b}{\le}&
CD\log^{1+\eps}\frac{LR\ell k}{D\gamma},
\eeq
where $D:=\sum_{i=1}^k v_i$,
(a) follows since $R/\gamma>1$ and assuming $L\geq 1$,
and
(b) follows by the concavity
of $x\log^{1+\eps}(u/x)$ (see Lemma~\ref{lem:log^eps} in Section \ref{sec:conv}).
We can assume $\ell\ge 2$ without loss of generality.
Combining the monotonicity of the covering number 
(see (\ref{eq:covn-pq})) and a lower bound on the covering number in terms of the fat-shattering dimension
(see Lemma \ref{lem:talag} in Section \ref{sec:cov-num-fat}) yields
\beq
\log\covn(G(\X_\ell),L_\infty(\mu_\ell),\gamma/2)
&\ge&
C\fat_{\gamma}(G)
=C\ell,
\eeq
whence
\beq
\ell
&\le&
CD\log^{1+\eps}\frac{LR\ell k}{D\gamma}.
\eeq

Using the elementary fact
\beq
x\le A\Log^{1+\eps} x
&\implies&
x\le c A\Log^{1+\eps} A
\qquad
x,A\ge1
\eeq
(with 
$x=LR\ell k/D\gamma$
and
$A=cLRk/\gamma$),
we get
\beq
\ell&\le& C D\Log^{1+\eps}\frac{LRk}{\gamma},
\eeq
which implies the claim.
\end{proof}

\subsection{Proof of Theorem~\ref{thm:max-fat-lin-R}}

We use the notation and results
from the Appendix,
and in particular, from
Section~\ref{sec:cov-num-lin}.

\begin{proof}[of Theorem~\ref{thm:max-fat-lin-R}]
A bound 
of this form for the $k$-fold maximum aggregation
was claimed in
\citet{kontorovich2018rademacher},
however the argument there
was flawed,
see Section~\ref{sec:discussion}.

 Let $G:\R^k\rightarrow \R$ be an aggregation rule that is $L$-Lipschitz in supremum norm, as defined in (\ref{def:lip}), and let $F_1,\ldots,F_k$ be bounded affine function classes, as defined in (\ref{eq:Fi-max-fat-lin-R}).
Suppose that some
$\X_\ell=\set{x_1,\ldots,x_\ell}\subset\X=B$
is $\gamma$-shattered by $G$,
let $F_i(\X_\ell)=\evalat{F_i}{\X_\ell}$,
and $\mu_\ell$ be the uniform distribution
on $\X_\ell$.
We upper bound the covering number as in Lemma~\ref{lem:ramon} (with $m=\ell$),
\beq
\log\covn(F_i(\X_\ell),L_\infty(\mu_\ell),\gamma)
&\le&
C\frac{R_i^2}{\gamma^2}\Log\frac{\ell \gamma^2}{R_i^2}
,
\qquad
0<\gamma<R_i.
\eeq
Denote $v_i:=L^2R_i^2/\gamma^2$, 
and consider the 
$L_\infty$
covering number
of
$F_i(\X_\ell)$
at
scale $\gamma/L$:
\beq
\log\covn(F_i(\X_\ell),L_\infty(\mu_\ell),\gamma/L)
&\le&
C v_i \Log\frac{\ell}{v_i}.
\eeq
Then Theorem~\ref{thm:max-cov} implies that 
\beq
\log\covn(G(\X_\ell),L_\infty(\mu_\ell),
\gamma/2)
&\le&
\sum_{i=1}^k
\log\covn(F_i(\X_\ell),L_\infty(\mu_\ell),
\gamma/2L)
\\
&\le&
C\sum_{i=1}^k
v_i \Log\frac{\ell}{v_i}
\\
&\maavar{a}{\le}&
C D \Log\frac{k\ell}{D},
\eeq
where $D:=\sum_{i=1}^k v_i$
and 
(a) follows by the concavity of $x\log(u/x)$
(see Corollary~\ref{cor:xlogn/x} in Section \ref{sec:conv}).
Combining the monotonicity of the covering number 
(see (\ref{eq:covn-pq})) and a lower bound on the covering number in terms of the fat-shattering dimension
(see Lemma \ref{lem:talag} in Section \ref{sec:cov-num-fat}) yields
\beq
\log\covn(G(\X_\ell),L_\infty(\mu_\ell),\gamma/2)
&\ge&
C\fat_{\gamma}(G)
=C\ell
,
\eeq
whence
\beq
\ell
&\le&
C D \Log\frac{k\ell}{D}.
\eeq
Using the elementary fact
\beq
x\le A\Log x
&\implies&
x\le cA\Log A,
\qquad
x,A\ge1
\eeq
(with $x=k\ell/D$
and
$A=ck$)
we get
$\ell\le cD\Log k$,
which implies the claim.

The result 
can easily be generalized to hinge-loss
affine classes. Let $F_i$ be an affine function class as in (\ref{eq:Fi-max-fat-lin-R}),
define $F'_i$ 
as the function class
on $B\times\set{-1,1}$
given by $F'_i=\set{(x,y)\mapsto yf(x); f\in F_i}$,
and 
the {\em hinge-loss affine class}
$F_i^{\hinge}$
as the function class
on $B\times\set{-1,1}$
given by $F_i^{\hinge}=\set{(x,y)\mapsto 
\max\set{0,1-f(x,y)}
; f\in F'_i}$.
One first observes that
the restriction of $F'_i$
to any $\set{(x_1,y_1),\ldots,(x_n,y_n)}$,
as a body in $\R^n$,
is identical to the 
the restriction of $F_i$
to
$\set{x_1,\ldots,x_n}$.
Interpreting $F_i^{\hinge}$ as a $2$-fold maximum
over the singleton class $H=\set{h\equiv0}$
and the bounded affine class $F'_i$ lets us invoke
Theorem~\ref{thm:max-cov}
to argue that $F_i$
and $F_i^{\hinge}$ have the same $L_\infty$ covering numbers.
Hence, the argument 
we deployed here to establish
\eqref{eq:Fi-max-fat-lin-R-bd}
for affine classes
also applies
to
$k$-fold
$L$-Lipschitz aggregations
hinge-loss classes.
\end{proof}

\subsection{Proof of Corollary~\ref{thm:rademacher}}
\begin{proof}[of Corollary~\ref{thm:rademacher}]
\citet[Theorem 7]{DBLP:journals/jmlr/RavivHO18}
upper-bounded the Rademacher complexity
of the maximum aggregation of $k$ hinge loss affine functions by $k/\sqrt n$. 

For $R_i$-bounded affine functions or hinge loss affine functions,
the analysis above, combined with
the calculation in
\citet{kontorovich2018rademacher}
yields a bound of
$
O\paren{\sqrt{\frac{\Log(k)\log^3(n)\sum_{i=1}^k 
{R_i^2}}{n}}}$. For completeness, we provide 
the
full proof.

Let $G_{\max}:\R^k\rightarrow \R$ be the $k$-fold maximum aggregation rule, as defined in (\ref{eq:Fmax}), and let $F_1,\ldots,F_k\subseteq\R^\Omega$ be an $R_i$-bounded affine functions as in (\ref{eq:Fi-max-fat-lin-R}) or a hinge loss affine functions as in (\ref{hinge}). Since this aggregation is $1$-Lipschitz in the supremum norm, 
Theorem \ref{thm:max-fat-lin-R} implies that
\beq
\fat_{\gamma}(G_{\max})
&\le&
\frac{
C\Log(k)
}{
\gamma^2
}
\sum_{i=1}^k 
{R_i^2}
,
\qquad
0<\gamma<\min_{i\in[k]}
R_i
,
\eeq
where $C>0$
is a universal constant. 

\paragraph{From fat-shattering to Rademacher.}
The fat-shattering estimate above
can be used to upper-bound the Rademacher complexity
by converting the former to a covering number bound
and plugging it into Dudley's chaining integral \citep{dudley1967sizes}:
\beqn
\label{eq:dudley-int}
\mathcal{R}_{n}(F) \le
\inf_{\alpha\ge0}\paren{4\alpha+12\int_\alpha^\infty\sqrt{\frac{\log \covn(t,F,\nrm{\cdot}_2)}{n}}dt},
\eeqn
where $\covn(\cdot)$ are the $\ell_2$ covering numbers (see Section~\ref{sec:cov-num}).

It remains to bound the covering numbers.
A simple way of doing so is to invoke
Lemmas 2.6, 3.2, and 3.3 in \citet{alon97scalesensitive}
--- but this incurs superfluous logarithmic factors in $n$.
Instead, we 
use the sharper estimate of
\citet{MR1965359},
stated here in Lemma \ref{lem:MV}.
Putting
$\tilde{R}=\max_i R_i$,
the latter yields
\beq
\mathcal{R}_{n}(G_{\max}) &\le&
\inf_{\alpha\ge0}\paren{4\alpha+12\int_\alpha^1\sqrt{\frac{\log \covn(t,G_{\max},\nrm{\cdot}_2)}{n}}dt}
\\&\le&
\inf_{\alpha\ge0}\paren{4\alpha+12c'\int_\alpha^1\sqrt{\frac{\fat_{ct/\tilde{R}}(G_{\max})\log\frac{2\tilde{R}}t}{n}}dt}
\\&\le&
\inf_{\alpha\ge0}\paren{4\alpha+{12c''}\sqrt{\frac{\Log(k)\sum_{i=1}^k 
{R_i^2}}{n}}
  \int_\alpha^1
  \frac{\tilde{R}}{t}
  \sqrt{
    \log\frac{2\tilde{R}}{t}
  }
dt}.
\eeq
Now
\beq
  \int_\alpha^1
  \frac{\tilde{R}}{t}
  \sqrt{
    \log\frac{2\tilde{R}}{t}
  }
dt
=
\frac{2\tilde{R}}{3}\paren{ \log(2\tilde{R}/\alpha)^{3/2}-(\log2\tilde{R})^{3/2}}
\eeq
and choosing $\alpha=1/\sqrt{n}$ yields
\beq
\mathcal{R}_{n}(G_{\max}) &\le&
\frac4{\sqrt n}
+12c''\sqrt{\frac{\Log(k)\sum_{i=1}^k 
{R_i^2}}{n}}\frac{2\tilde{R}}{3}\paren{
  \log(2\tilde{R}\sqrt n)^{3/2}
  -
(\log 2\tilde{R})^{3/2}
}
\\
&=&
O\paren{\sqrt{\frac{\Log(k)\log^3(\tilde{R}n)\tilde{R}^2\sum_{i=1}^k 
{R_i^2}}{n}}}.
\eeq
\end{proof}

\subsection{Proof of Theorem~\ref{thm:max-fat-lin-d}}

\begin{proof}[of Theorem~\ref{thm:max-fat-lin-d}]
Let $G_{\max}:\R^k\rightarrow \R$ be the $k$-fold maximum aggregation rule, as defined in (\ref{eq:Fmax}), and let $F_1,\ldots,F_k\subseteq\R^\Omega$ be real-valued function classes. 
Note that $G_{\max}$ is an aggregation that commutes with shift, as defined in (\ref{def:commutes-shift}).

By (\ref{eq:fatfat0}),
there is an $r\in\R^\X$
such that $\fat_\gamma(G_{\max})
=\faat_\gamma(G_{\max}-r)
$.
As in (\ref{eq:Fmax'}),
put $F'_i:=F_i-r$
and
$
G'_{\max}:=G_{\max}-r
=G_{\max}(F'_1,\ldots,F'_k)
$.
Define $\bar G_{\max}=\sign(G')$
and
$\bar F_i=\sign(F_i')$.

Since $\sign$
and $\max$ commute, we have
$\bar G_{\max}
=\max(\bar F_{i\in[k]})
$.
We claim that
\beqn
\label{eq:faatmax}
\faat_\gamma(G_{\max}')
&\le&
\vc(\bar G_{\max}).
\eeqn
Indeed, any $S\subset\X$
that is $\gamma$-shattered
with shift $r=0$ by any $G\subset\R^\X$
is also VC-shattered by $\sign(G)$.
(See Section~\ref{sec:max-fat},
and
notice that the converse implication---and the reverse inequality---do not hold.)
It holds that
\beq
d+1
\overset{(a)}{=}
\vc(\bar F_i)
\overset{(b)}{=}
\faat_\gamma(F_i)
\overset{(c)}{=}
\fat_\gamma(F_i)
\overset{(d)}{=}
\fat_\gamma(F_i'),
\eeq
where (a) follows from a standard argument
(e.g., \citet[Example 3.2]{mohri-book2012}), 
(b) 
holds because
any $S\subset\R^d$
that is VC-shattered by
$\sign(F'_i)$
is also
$\gamma$-shattered
by $F'_i$
with shift $r=0$, (c)
follows from Lemma
\ref{lem:fat-faat}, since the class is closed under scalar multiplication, and (d) holds since the shattering remains the same for the shifted class.

Now the argument of
\citet[Lemma 3.2.3]{MR1072253}
applies:
\beqn
\label{eq:blumer}
\vc(\bar G_{\max})
&\le&
2(d+1)k\log(3k)
\eeqn
(this holds for any $k$-fold aggregation function,
not just the maximum).
Combining (\ref{eq:faatmax})
with (\ref{eq:blumer}) proves the claim.
\end{proof}

\subsection{Proof of Theorem~\ref{thm:logk-lb}}

\begin{proof}[of Theorem~\ref{thm:logk-lb}]
It follows from
\citet[Example 3.2]{mohri-book2012}
that $\vc(\sign(F_i))=d+1$.
Since $F_i$ is closed under
scalar multiplication,
a scaling argument
shows that 
any $S\subset\R^d$
that is VC-shattered by
$\sign(F_i)$
is also
$\gamma$-shattered
by $F_i$
with shift $r=0$,
whence
$\faat_\gamma(F_i)
=d+1$
for all $\gamma>0$;
invoking
Lemma~\ref{lem:fat-faat}
extends this to $\fat_\gamma(F_i)$
as well.
Now
\citet[Theorem 1]{DBLP:journals/jmlr/CsikosMK19}
shows that the $k$-fold unions
of 
half-spaces
necessarily shatter some set $S\subset\R^d$
of size at least
$cdk\log k$.
Since union is a special case of 
the max operator,
and the latter commutes with $\sign$,
the scaling argument shows that
this
$S$ is $\gamma$-shattered
by
$G_{\max}$
with shift $r=0$.
Hence,
$\fat_\gamma(G_{\max})
\ge
\faat_\gamma(G_{\max})
\ge
|S|
$,
which proves the claim.
\end{proof}

\section{Discussion}\label{sec:discussion}
In this paper, we proved upper and lower bounds on the fat-shattering dimension of aggregation rules as a function of the fat-shattering dimension of the component classes. We leave some remaining gaps for future work.
First, for aggregation rules that commute with shifts, assuming
    $\fat_\gamma(F_i)\leq d$, for $1\leq i \leq k$, we show in \Cref{thm:max-fat} that
    \begin{align*}
    \fat_\gamma(G(F_1,\ldots,F_k))
    \leq Cdk\log^2\left(dk\right),
    \qquad
    \gamma>0,
    \end{align*} 
$C>0$ is a universal constant.
We pose the following 
\paragraph{Open problem.}
\label{conj:holy-grail}
Let
$G$ be an aggregation rule that commutes with shifts.
Is it the case that
for all $F_i\subseteq\R^\X$
with
$\fat_\gamma(F_i)\leq d$, $i\in[k]$,
we have
\beq
 \fat_\gamma(G(F_1,\ldots,F_k))
&\le&
Cdk\log\left(k\right)
,
\qquad
\gamma>0,
\eeq
for some universal $C>0$?

In light of Theorem~\ref{thm:logk-lb},
this is the best one could
hope for in general.
We pose also the following conjecture about bounded affine functions.
\begin{conjecture}
\label{conj:fat-lin-R-tight}
Theorem~\ref{thm:max-fat-lin-R} is tight up to constants. 
For 
$R_i$-bounded affine functions
and
aggregation rule $G$ that is 1-Lipschitz in supremum norm,
\beqn
\fat_{\gamma}(G(F_1,\ldots,F_k))
&\ge&
\frac{
C\Log(k)
}{
\gamma^2
}
\sum_{i=1}^k 
{R_i^2}
,
\qquad
0<\gamma<\min_{i\in[k]}
R_i
,
\eeqn
where $C>0$
is a universal constant.

\end{conjecture}

Throughout the paper,
we mentioned several mistaken claims in the literature.
In this section, we briefly discuss the nature
of these mistakes---which are, in a sense,
variations on the same kind of error.
We begin with
\citet[Lemma 14]{DBLP:conf/alt/AttiasKM19},
which incorrectly claimed that
any partial function class $F\st$
has a disambiguation $\bar F$
such that $\vc(\bar F)\le\vc(F\st)$
(see Section~\ref{sec:max-fat}
for the definitions).
The mistake was pointed out to us by
Yann Guermeur,
and later,
\citet[Theorem 11]{DBLP:journals/corr/abs-2107-08444}
showed
that there exist partial classes
$F\st$ with $\vc(F\st)=1$
for which every disambiguation
$\bar F$ has $\vc(\bar F)=\infty$.

\citet{kontorovich2018rademacher}
attempted to prove the bound
stated in our Theorem~\ref{thm:max-fat-lin-R}
(up to constants, and only for linear classes).
The argument proceeded via a reduction to
the Boolean case, 
as in our proof of Theorem~\ref{thm:max-fat-lin-d}.
It was correctly observed that
if, say, 
some finite $S\subset\X$
is $1$-shattered
by
$F_i$ with shift $r=0$,
then it is also VC-shattered 
by $\sign(F_i)$. Neglected was the fact
that $\sign(F_i)$ might shatter additional points in $\X\setminus S$---and,
in sufficiently high dimension, it necessarily will.
The crux of the matter is that
(\ref{eq:faatmax}) holds
in the dimension-dependent but
not the dimension-free setting;
again, this may be seen as a variant
of the disambiguation mistake.

Finally, the proof of
\citet[Lemma 6]{DBLP:journals/tcs/HannekeK19}
claims,
in the first display, 
that the shattered set can be classified
with large margin, which is incorrect
--- yet another variant of mistaken
disambiguation.

\newpage

\acks{We thank 
Steve Hanneke
and
Ramon van Handel for very helpful
discussions;
the latter, in particular,
patiently explained to us how to prove
Lemma~\ref{lem:ramon}.
Roman Vershynin kindly gave us permission
to share his example in Remark~\ref{rem:vers}.
This research was partially supported by
the Israel Science Foundation
(grant No. 1602/19), an Amazon Research Award,
the Ben-Gurion University Data Science Research Center,
and
Cyber Security Research Center, Prime Minister's Office.
}


\newpage

\appendix

\section{Auxiliary results}

\subsection{Covering numbers and Lipschitz Aggregations}
\label{sec:cov-num}

\begin{lemma}\label{lem:lp-lip}
If $G:\R^k \rightarrow \R$ is $L$-Lipschitz under $\nrm{\cdot}_p$, then $G: (\R^{\Omega})^k\rightarrow \R^{\Omega}$ is $L$-Lipschitz in $\nrmpk{\cdot}{p}$.
\end{lemma}
\begin{proof}
    \begin{align*}
        \nrmp{G(f_1,\ldots,f_k)-G(f'_1,\ldots,f'_k)}{p}^p
        &=
        \int_\X |G(f_1,\ldots,f_k)(x)-G(f'_1,\ldots,f'_k)(x)|^p\d\mu(x)
        \\
        &=
        \int_\X |G(f_1(x),\ldots,f_k(x))-G(f'_1(x),\ldots,f'_k(x))|^p\d\mu(x)
        \\
        &\leq
        \int_\X L^p\nrm{(f_1(x),\ldots,f_k(x))-(f'_1(x),\ldots,f'_k(x))}_p^p \d\mu(x),
    \end{align*}
where the inequality follows from the assumption that $G:\R^k\rightarrow \R$ is $L$-Lipschitz in $\nrm{\cdot}_p$. This proves
\begin{align*}
        \nrmp{G(f_1,\ldots,f_k)-G(f'_1,\ldots,f'_k)}{p}
        \leq
        L \nrmpk{(f_1,\ldots,f_k)-(f'_1,\ldots,f'_k)}{p},
    \end{align*}
and hence the claim.
\end{proof}
\begin{proof}[of \Cref{thm:max-cov}]
Suppose $p<\infty$, and let $g= G(f_1,\ldots,f_k) \in G(F_1,\ldots,F_k)$.
For each $i\in[k]$,
let $\hat F_i\subset F_i$
be a 
$t/Lk^{1/p}$-cover of $F_i$.
Let each $f_i$ be ``$t/Lk^{1/p}$-covered''
by
some $\hat f_i\in\hat F_i$,
in the sense that
$\nrmp{f_i-\hat f_i}p\le t/Lk^{1/p}$.
Assuming that $G:\R^k\rightarrow \R$ is $L$-Lipschitz in $\nrm{\cdot}_p$, Lemma \ref{lem:lp-lip} implies that $G: (\R^{\Omega})^k\rightarrow \R^{\Omega}$ is $L$-Lipschitz in $\nrmpk{\cdot}{p}$.
Then it follows that $g$ is $t$-covered by $G(\hat f_1,\ldots,\hat f_k)$, since 

\begin{align*}
        \nrmp{G(f_1,\ldots,f_k)-G(f'_1,\ldots,f'_k)}{p}^p
        &\leq
        L^p \nrmpk{(f_1,\ldots,f_k)-(f'_1,\ldots,f'_k)}{p}^p
        \\
        &=
        L^p \int_\X \nrm{(f_1(x),\ldots,f_k(x))-(f'_1(x),\ldots,f'_k(x))}_p^p \d\mu(x)
        \\
        &=
        L^p \int_\X \sum^k_{i=1}\abs{f_i(x)-f'_i(x)}^p \d\mu(x)
        \\
        &=
        L^p \sum^k_{i=1} \int_\X \abs{f_i(x)-f'_i(x)}^p \d\mu(x)
        \\
        &=
        L^p \sum^k_{i=1} \nrmp{f_i-f'_i}{p}^p
        \\
        &\leq
        L^p k\left(\frac{t}{Lk^{1/p}}\right)^p
        \\
        &=
        t^p,
    \end{align*}
and so $\nrmp{G(f_1,\ldots,f_k)-G(f'_1,\ldots,f'_k)}{p}\leq t$.

We conclude
that
$G(F_1,\ldots,F_k)$ has a $t$-cover of size
$|\hat F_1\times\hat F_2\times\ldots\times\hat F_k|$,
which proves the claim.
The case $p=\infty$ is proved analogously (or, alternatively, as a limiting case of $p<\infty$).
\end{proof}
We show that natural aggregations are Lipschitz in $\nrm{\cdot}_p$ norms, $p\in [1,\infty)$, and in supremum norm. 
The following facts are elementary:
\beqn
\label{eq:abcd-max}
|a\vee b - c\vee d|
&\le&
|a-c|
\vee
|b-d|,
\qquad
a,b,c,d\in\R; \\
\label{eq:abcd-min}
|a\wedge b - c\wedge d|
&\le&
|a-c|
\vee
|b-d|,
\qquad
a,b,c,d\in\R,
\eeqn
where $s\vee t:=\max\set{s,t}$
and $s\wedge t:=\min\set{s,t}$.

\begin{lemma}[Maximum aggregation is 1-Lipschitz]
\label{lem:ff'}
Let $G_{\max}:\R^k\rightarrow \R$ be the maximum aggregation, then for any $x,x'\in \R^k$ and 
$p\in[1,\infty]$,
\beq
\abs{G(x)-G(x')}
\leq 
\nrm{x-x'}_p.
\eeq
\end{lemma}
\begin{proof}
For $k=2$ and $p=\infty$, the claim follows from 
the stronger,
pointwise inequality
\eqref{eq:abcd-max}.
The proof follows by simple induction on $k$. Since $\nrm{\cdot}_\infty \leq \nrm{\cdot}_p$, we conclude the proof for $p\in [1,\infty]$.
\end{proof}

\begin{lemma}[Max-Min aggregation is 1-Lipschitz]
\label{lem:max-min}
If $G_{\maxmin}:\R^k\rightarrow \R$ 
is the max-min aggregation, 
then for any $x,x'\in \R^k$ and 
$p\in[1,\infty]$,
\beq
\abs{G(x)-G(x')}
\leq 
\nrm{x-x'}_p.
\eeq
\end{lemma}
\begin{proof}
The inequalities
\eqref{eq:abcd-max},
\eqref{eq:abcd-min}
imply
that
the $k$-fold max and min
aggregations are both $1$-Lipschitz
with respect to $\nrm{\cdot}_\infty$.
Hence, for all $x,y\in\R^{k\times\ell}$, we have
\beq
\abs{
\min_{j\in[\ell]}x_{ij}
-
\min_{j\in[\ell]}y_{ij}
}
\le
\max_{j\in[\ell]}\abs{x_{ij}-x_{ij}},
\qquad i\in[k]
\eeq
and further,
\beq
\abs{
\max_{i\in[k]}\min_{j\in[\ell]}x_{ij}
-
\max_{i\in[k]}\min_{j\in[\ell]}y_{ij}
}
\le
\max_{i\in[k]}\max_{j\in[\ell]}\abs{x_{ij}-x_{ij}}.
\eeq
This proves that
$
\abs{G(x)-G(x')}
\leq 
\nrm{x-x'}_\infty
$.
Since $\nrm{\cdot}_\infty \leq \nrm{\cdot}_p$, 
the claim holds 
for all $p\in [1,\infty]$.
\end{proof}

\subsection{Covering numbers and the fat-shattering dimension}
\label{sec:cov-num-fat}

In this section, we summarize some known results connecting the covering numbers of a bounded function class to its fat-shattering dimension.

\begin{lemma}[\citet{talagrand2003vapnik}, Proposition 1.4]
\label{lem:talag}
For any $F\subseteq[-R,R]^\X$,
there exists a probability measure
$\mu$ on $\X$
such that
\beqn
\label{eq:talag}
\covn(F,L_2(\mu),t)
&\ge&
2^{C\fat_{2t}(F)},
\qquad
0<t<R,
\eeqn
where $C>0$
is a
universal constants.
Moreover,
$\mu$ may be taken to be the uniform
distribution on
any $2t$-shattered subset of $\X$.
\end{lemma}
{\bf Remark.}
The tightness of (\ref{eq:talag})
is trivially demonstrated by the example
$F=\set{-\gamma,\gamma}^n$.

\begin{lemma}[\citet{MR1965359}, Theorem 1]
\label{lem:MV}
For all $F\subseteq[-1,1]^\X$
and all probability measures $\mu$,
\beqn
\label{eq:MV}
\covn(F,L_2(\mu),t)
&\le&
\paren{\frac{2}{t}}^{C\fat_{c t}(F)},
\qquad
0<t<1,
\eeqn
where $C,c>0$
are universal constants.
\end{lemma}
\begin{remark}
\label{rem:vers}
The following example
due to
\citet{vers-private} shows that
(\ref{eq:MV}) 
is tight.
Take $\X=[n]$ and $F=[-1,1]^\X$.
Then, for all sufficiently small $t>0$,
we have $\fat_t(F)=n$. However, a simple
volumetric calculation shows that
$\covn(F,t)$ behaves as $(C/t)^n$
for small $t$, where $C>0$ is a constant.
\end{remark}

\begin{lemma}[\citet{Rudelson-Vershynin06}]
\label{lem:RV}
Suppose that $p\in[2,\infty)$,
$\mu$ is a probability measure on $\X$,
and $R>0$.
If
$F\subset L_p(\X,\mu)$
satisfies $\sup_{f\in F}\nrmp{f}{2p}\le R$,
then
\beq
\label{eq:RVp}
\log\covn(F,L_p(\mu),t)
&\le&
Cp^2\fat_{ct}(F)\log\frac{R}{ct},
\qquad
0<t<R;
\eeq
furthermore, for all $\eps>0$,
if
$\sup_{f\in F}\nrmp{f}{\infty}\le R$,
then
\beq
\label{eq:RV-infty}
\log\covn(F,L_\infty(\mu),t)
&\le&
Cv\log(Rn/vt)\log^\eps(n/v),
\qquad
0<t<R,
\eeq
where $n=|\X|$,
$v=\fat_{c\eps t}(F)$,
and
$C,c>0$
are universal constants.

\end{lemma}

\subsection{Covering numbers of linear and affine classes}
\label{sec:cov-num-lin}
Let $B\subset\R^d$ be the $d$-dimensional
Euclidean unit ball 
and
\beq
\label{eq:affineF}
F=\set{x\mapsto w\cdot x+b;
\nrm{w}\vee |b|\le R}
\eeq
be the collection of $R$-bounded affine functions
on $\X=B$.
\begin{remark}
\label{rem:lin-vs-aff}
There is a trivial reduction from
an $R$-bounded affine class in $d$
dimensions to a $2R$-bounded {\em linear}
class in $d+1$ dimensions, via the standard
trick of adding an extra dummy dimension.
This only affects the covering number bounds
up to constants.
\end{remark}
For $\X_n\subset B$,
$|\X_n|=n$,
define 
$
F(\X_n)=\evalat{F}{\X_n}
$,
and
endow $\X_n$ with the uniform measure
$\mu_n$. 
\citet[Theorem 4]{zhang2002covering} implies
the covering number estimate
\beq
\label{eq:zhang-infty}
\log\covn(F(\X_n),L_\infty(\mu_n),t)
&\le&
C\frac{R^2}{t^2}\Log\frac{nR}{t},
\qquad
t>0,
\eeq
where $C>0$ is a universal constant
(Zhang's result is more general
and allows to compute explicit constants).
We will use the following sharper bound:
\begin{lemma}
\label{lem:ramon}
\beq
\label{eq:ramon}
\log\covn(F(\X_n),L_\infty(\mu_n),t)
&\le&
C\frac{R^2}{t^2}\Log\frac{mt^2}{R^2},
\qquad
0<t<R,
\eeq
where 
$m=\min\set{n,d}$
and
$C>0$ is a universal constant.
\end{lemma}
\begin{proof}
The result is folklore knowledge, but
we provide a proof for completeness.

We argue that there is no loss of generality
in assuming
$d\ge n$.
Indeed, if $n>d$,
then $X$
is spanned by some $X'=\set{x_1',\ldots,x_d'}\subset B$
and $F\subset\operatorname{span}(X')$
is also a $d$-dimensional set. Thus,
we assume $d\ge n$ henceforth.
Via a standard infinitesimal perturbation,
we can assume that $X$ is a linearly
independent set (i.e., spans $\R^m$).
If we treat $X$ as an $m\times d$
matrix, then $F=XB$, which means that
$F$ is an ellipsoid.
We are interested in $\ell_\infty$
covering numbers of $F$.

Let $K\subset\R^d$ be such that
$XK=L$, where $L=B_\infty^m$ is
the unit cube. 
(The existence of a $K$
such that 
$XK\subset L$
is obvious, but because we assumed that
$X$ spans $\R^m$, every point in $[-1,1]^m$
has a pre-image under $X$.)
Let us compute the polar
body $K^\circ$, defined as
\beq
K^\circ
=\set{
u\in\R^d:
\sup_{v\in K}v\cdot u\le 1
}.
\eeq
We claim that
\beq
K^\circ
=
\absconv(X)
=:
\set{
\sum_{i=1}^m\alpha_i x_i;
\sum \abs{\alpha_i}\le1
}.
\eeq

Indeed, consider a
$
z=\sum_{i=1}^m\alpha_i x_i
\in 
\absconv(X)
$. Then, for any $v\in K$, we have
\beq
v\cdot z
&=&
v\cdot
\sum_{i=1}^m\alpha_i x_i
\\
&=&
\sum_{i=1}^m\alpha_i (v\cdot x_i)
=
\sum_{i=1}^m|\alpha_i|\le1
\qquad
\implies z\in K^\circ
,
\eeq
where we have used
$|v\cdot x_i|\le 1$
(since $XK=L=B_\infty^m=[-1,1]^m$)
and H\"older's inequality.
This shows that
$
\absconv(X)
\subseteq
K^\circ
$.
On the other hand, consider
any $u\in K^\circ$. There is no
loss of generality in
assuming that $u$ is spanned by $X$,
that is,
$u=\sum_{i=1}^m\alpha_i x_i$,
for $\alpha_i\in\R$.
By definition of $u\in K^\circ$, we have
\beq
\sup_{v\in K}v\cdot u
&=&
\sup_{v\in K}v\cdot \sum_{i=1}^m\alpha_i
x_i
=
\sup_{v\in K} \sum_{i=1}^m\alpha_i
(v\cdot x_i)
\le 1.
\eeq
Now because $XK=[-1,1]^m$,
for each choice of $\alpha\in\R^m$,
there is a $v\in K$ such that
$|v\cdot x_i|=\sign(\alpha_i)$
for all $i\in[m]$. This shows that
we must have $\sum_{i=1}^m|\alpha_i|\le 1$,
and proves 
$
K^\circ
\subseteq
\absconv(X)
$.

It is well-known
(and easy to verify)
that covering numbers enjoy an affine invariance:
\beq
N(F,L) := N(XB,XK) = N(B,K),
\eeq
where $N(A,B)$,
for two sets $A,B$,
is the smallest number of copies of
$B$ necessary to cover $A$.
Now the seminal result of
\citet{Artstein-duality-entropy}
applies:
for all $t>0$,
\beq
\log N(B,tK)
\le
a
\log N(K^\circ,bt B),
\eeq
where $a,b>0$ are universal constants.

This reduces the problem to estimating the
$\ell_2$-covering numbers of
$\absconv(X)$.
The latter may be achieved 
via Maurey's method 
\citep
[Corollary 0.0.4
and
Exercise 0.0.6]
{MR3837109}:
the $t$-covering number
of
$\absconv(r X)$
under $\ell_2$
is at most
\beq
(c+cm t^2/r^2)^{\ceil{r^2/ t^2}},
\eeq
where $c>0$ is a universal constant.

\end{proof}

\subsection{Fat-shattering dimension of linear and affine classes}
\label{sec:affine}
In this section, $\X=\R^d$
and $B\subset\R^d$ denotes the Euclidean unit ball.
A function $f:\X\to\R$
is said to be {\em affine}
if it is of the form $f(x)=w\cdot x+b$,
for some $w\in\R^d$ and $b\in\R$,
where $\cdot$ denotes the Euclidean inner product. 

Throughout the paper, we have have referred to
$R$-bounded affine function classes as those
for which $\nrm{w}\vee |b|\le R$.
In this section, we define the larger class of
{\em $R$-semi-bounded} affine functions, as those
for which $\nrm{w}\le R$, but $b$ may be unbounded.
In particular, the covering-number results
(and the reduction to linear classes
spelled out in Remark~\ref{rem:lin-vs-aff})
do not apply to semi-bounded affine classes.

The following simple result may be of independent
interest.
\begin{lemma}
\label{lem:fat-faat}
Let $F\subset\R^\X$ be some
collection of 
functions
with the 
closure
property
\beqn
\label{eq:closure}
f,g\in F &\implies
(f-g)/2\in F.
\eeqn
Then, for all $\gamma>0$, we have
$
\fat_\gamma(F)=\faat_\gamma(F)
$
.
\end{lemma}
\begin{proof}
\newcommand{\upw}{{\hat{w}}}
\newcommand{\dnw}{{\check{w}}}
\newcommand{\upb}{{\hat{b}}}
\newcommand{\dnb}{{\check{b}}}
\newcommand{\upf}{{\hat{f}}}
\newcommand{\dnf}{{\check{f}}}

Suppose that some set $\set{ x_1,\ldots, x_k}$
is $\gamma$-shattered by $F$.
That means that there is an $r
\in\R^k
$
such that for all $y\in\set{-1,1}^k$,
there is a 
$f=f_y\in F$
for which
\beqn
\label{eq:ti-r}
\gamma&\le& y_i( 
f(x_i)
-r_i
),
\qquad i\in[k].
\eeqn
Now for any $y\in\set{-1,1}^k$,
let
$
\upf=f_y
$
and
$
\dnf=f_{-y}
$.
Then, for each $i\in[k]$, we have
\beq
\gamma
&\le& 
\phantom{-}
y_i(
\upf(x_i)
-r_i) 
,\\
\gamma
&\le& 
-y_i(
\dnf(x_i)
-r_i) 
.
\eeq
It follows that 
$
f=(\upf-\dnf)/2
$
achieves (\ref{eq:ti-r}), for the given $y$,
with $r\equiv 0$.
Now (\ref{eq:closure}) implies that
the function defined by $
f
$
belongs to $F$, which completes the proof.
\end{proof}

Now is well-known \citep[Theorem 4.6]{299098} that bounded {\em linear} functions
---
i.e., function classes 
on $B$
of the form 
$F=\set{x\mapsto 
w\cdot x; 
\nrm{w}\le R
}$,
also known as
{\em homogeneous hyperplanes}
---
satisfy $\fat_\gamma(F)\le (R/\gamma)^2$.
The discussion in
\citet[p. 102]{DBLP:journals/tcs/HannekeK19}
shows that the 
common approach of reducing
of the general (affine) case to the linear (homogeneous, $b=0$) case, via the addition of a ``dummy'' coordinate, incurs a large suboptimal factor in the bound.
\citet[Lemma 6]{DBLP:journals/tcs/HannekeK19}
is essentially an analysis of the fat-shattering
dimension of bounded affine functions.
Although this result contains a mistake (see Section~\ref{sec:discussion}),
much of the proof
technique can be salvaged:
\begin{lemma}
\label{lem:fat-hyp}
The 
semi-bounded affine
function class
on $B$
defined by
$F=\set{x\mapsto w\cdot x+b;
\nrm{ w}\le R
}$
in $d$ dimensions
satisfies
\beq
\fat_\gamma(F)
&\le&
\min\set{d+1,
\paren{\frac{\left(1+\sqrt{\frac{8}{\pi}}\right)R}{\gamma}}^2
},
\qquad 0<\gamma\le R.
\eeq
\end{lemma}
\begin{proof}
Since $F$ satisfies (\ref{eq:closure}),
it suffices to consider $\faat_\gamma(F)$,
and so the shattering
condition 
simplifies to
\beqn
\label{eq:ti}
\gamma&\le& y_i( w\cdot x_i+b),
\qquad i\in[k].
\eeqn
Now $\faat_\gamma(F)$
is 
always
upper-bounded by the
VC-dimension of the corresponding
class thresholded at zero,
i.e., $\sign(F)$.
For $d$-dimensional inhomogeneous hyperplanes, the latter is exactly $d+1$
\citep[Example 3.2]{mohri-book2012}.
Having dispensed with the dimension-dependent part in the bound,
we how focus on the 
$R$-dependent one.

Let us
observe,
as in \citet[Lemma 6]{DBLP:journals/tcs/HannekeK19},
that for $\nrm{ x_i}\le1$
and $\nrm{ w}$, $\gamma\le R$,
one can always realize (\ref{eq:ti})
with $|b|\le2R$; which is what we shall
assume, without generality, henceforth.
Summing up the $k$ inequalities 
in (\ref{eq:ti})
yields
\beq
k\gamma &\le&
 w\cdot\sum_{i=1}^k y_i x_i + b\sum_{i=1}^ky_i
\le
R\nrm{
\sum_{i=1}^k y_i x_i
}
+
2R\abs{
\sum_{i=1}^ky_i
}.
\eeq
Letting $y$ be drawn uniformly from $\set{-1,1}^k$
and taking expectations, we have
\beq
k\gamma &\le&
R\E\nrm{
\sum_{i=1}^k y_i x_i
}
+
2R\E\abs{
\sum_{i=1}^ky_i
}
\le
R\sqrt{\E\nrm{
\sum_{i=1}^k y_i x_i
}^2}
+
2R\sqrt{\E\paren{
\sum_{i=1}^ky_i
}^2}
\\
&=&
R\sqrt{\sum_{i=1}^k\nrm{
  x_i
}^2}
+
2R\sqrt{\sum_{i=1}^k
y_i
^2}
\le
3R\sqrt{k}.
\eeq
Isolating $k$ on the left-hand side
of the inequality proves the claim $k\leq \left(\frac{3R}{\gamma}\right)^2$.

Following a referee's suggestion,
we improve the constant as follows. Note that
\beq
\E\abs{
\sum_{i=1}^ky_i
}
=
\frac{1}{2^k}\sum^k_{i=0}\binom{k}{i}|k-2i|
=
\frac{k}{2^{k-1}}
\binom{k-1}{\floor{\frac{k}{2}}}
\leq 
\sqrt{\frac{2}{\pi}}\frac{k}{\sqrt{k+\frac{1}{2}}},
\eeq 
where the inequality follows from binomial coefficient estimate via Stirling's approximation.
Thus,
\beq
k\gamma &\le&
R\sqrt{k}+2R\sqrt{\frac{2}{\pi}}\frac{k}{\sqrt{k+\frac{1}{2}}}
\le
R\sqrt{k}+2R\sqrt{\frac{2}{\pi}}\sqrt{k},
\eeq
which proves that $k\leq \left(\frac{\left(1+\sqrt{\frac{8}{\pi}}\right)R}{\gamma}\right)^2$.
\end{proof}

\subsection{Concavity
miscellanea}
\label{sec:conv}

The results below are routine exercises in
differentiation and Jensen's inequality.

\begin{lemma}
For $u>0$,
the
function
$x\mapsto x\log(u/x)$
is concave on $(0,\infty)$.
\end{lemma}
\begin{corollary}
\label{cor:xlogn/x}
For all $u>0$
and $v_i>0$, $i\in[k]$,
\beq
\sum_{i=1}^k v_i\log(u/v_i)
&\le&
\paren{\sum v_i}\log\frac{uk}{\sum v_i}.
\eeq
\end{corollary}

\begin{lemma}
\label{lem:log^eps}
For 
$0\le\eps\le\log2$
and
$u\ge2$,
the
function
$x\mapsto x\log^{1+\eps}(u/x)$
is concave on $[1,\infty)$.
It follows that
for $\eps,u$ as above
and $v_i\ge1$, $i\in[k]$,
\beq
\sum_{i=1}^k v_i\log^{1+\eps}(u/v_i)
&\le&
\paren{\sum v_i}\log^{1+\eps}\frac{uk}{\sum v_i}.
\eeq
\end{lemma}

\end{document}